\documentclass[11pt,twoside]{amsart}
\usepackage{amsmath, amsthm, amscd, amsfonts, amssymb, graphicx, color}
\usepackage[bookmarksnumbered, colorlinks, plainpages]{hyperref}

\newtheorem{thm}{Theorem}[section]

\newtheorem{prop}[thm]{Proposition}
\newtheorem{defn}[thm]{Definition}

\numberwithin{equation}{section}
\def\pn{\par\noindent}

\begin{document}



\title{Double Total Domination in Harary Graphs}
\author{Adel P. Kazemi and Behnaz Pahlavsay
}

\thanks{{\scriptsize
\hskip -0.4 true cm MSC(2010): 
05C69.
}}

\maketitle

\begin{abstract}
Let $G$ be a graph with minimum degree at least 2. A set $D\subseteq
V$ is a double total dominating set of $G$ if each vertex is
adjacent to at least two vertices in $D$. The double total
domination number $\gamma _{\times 2,t}(G)$ of $G$ is the minimum
cardinality of a double total dominating set of $G$. In this paper,
we will find double total domination number of Harary graphs.
\end{abstract}
\vskip 0.2 true cm


\pagestyle{myheadings}
\markboth{\rightline {\scriptsize  A. P. Kazemi, B. Pahlavsay}}
         {\leftline{\scriptsize Double Total Domination in Harary Graphs}}

\bigskip
\bigskip

\section{introduction}
\vskip 0.4 true cm

Let $G$ be a simple graph with the \emph{vertex set} $V=V(G)$ and
the \emph{edge set} $E=E(G)$. The \emph{order} $\mid V\mid $ and
\emph{size} $\mid E\mid $ of $G$ are denoted by $n=n(G)$ and
$m=m(G)$, respectively. The \emph{open neighborhood} and the
\emph{closed neighborhoods} of a vertex $v\in V$ are
$N_{G}(v)=\{u\in V\mid uv\in E\}$ and $N_{G}[v]=N_{G}(v)\cup \{v\}$,
respectively. The \emph{degree} of a vertex $v\in V$ is $deg(v)=\mid N(v)\mid $. The
\emph{minimum} and \emph{maximum degree} of a graph $G$ are denoted
by $\delta =\delta (G)$ and $\Delta =\Delta (G)$, respectively.

\vspace{0.25 cm}

The research of domination in graphs has been an evergreen of the
graph theory. Its basic concept is the dominating set and the
domination number. The literature on this subject has been surveyed
and detailed in the two books by Haynes, Hedetniemi, and
Slater~\cite{HHS1, HHS2}. And many variants of the dominating set
were introduced and the corresponding numerical invariants were
defined for them. For example, the $k$-tuple total domination number
is defined in \cite{HK10} by Henning and Kazemi, which is an extension
of the total domination number (for more information see
\cite{HK11,Kaz}).

\begin{defn} \cite{HK10}
\emph{Let $k\geq 1$ be an integer and let $G$ be a graph with
$\delta(G)\geq k$. A subset $S\subseteq V(G)$ is called a} $k$-tuple
total dominating set, \emph{briefly kTDS, of $G$, if for each $x\in
V(G)$, $\mid N(x)\cap S\mid \geq k$. The minimum number of vertices
of a $k$-tuple total dominating set of a graph $G$ is called the}
$k$-tuple total domination number $\gamma _{\times k,t}(G)$ \emph{of $G$}.
\end{defn}

The $2$-tuple total dominating set and the $2$-tuple total
domination number are known as \emph{double total dominating set} and
\emph{double total domination number}, respectively.
\vspace{0.25 cm}

\textbf{Harary graphs:} \cite{West} Given $m\leq n$, place $n$ vertices $1$, $2$, $...$, $n$ around a
circle, equally spaced. If $m$ is even, form $H_{m,n}$ by making
each vertex adjacent to the nearest $\frac{m}{2}$ vertices in each direction around the circle. If $m$ is odd and $n$ is even, form $H_{m,n}$\ by making each vertex adjacent to the nearest $\frac{m-1}{2}$ vertices in each direction and to the diametrically opposite vertex.
In each case, $H_{m,n}$ is $m$-regular. When $m$ and $n$ are both odd,
index the vertices by the integers modulo $n$. Construct $H_{m,n}$ from $H_{m-1,n}$ by adding the edges
$i\leftrightarrow i+\frac{n-1}{2}$ \ for $0\leq i\leq \frac{n-1}{2}$.

\vspace{0.25 cm}

Here, we find the double domination number of Harary graphs $H_{m,n}$. The next propositions are useful for our investigations.

\begin{prop}
\label{HK} \emph{(\textbf{Henning, Kazemi \cite{HK10} 2010)}} For any
graph $G$ of order $n$ with $\delta (G)\geq k$, we have

\emph{i}. $\max \{\gamma _{\times k}(G),k+1\}\leq \gamma _{\times
k,t}(G)\leq n$,

\emph{ii}. if $G$ is a spanning subgraph of a graph $H$, then $\gamma
_{\times k,t}(H)\leq \gamma _{\times k,t}(G)$,

\emph{iii}. For any vertex $v$ of degree $k$, $N_{G}(v)$ is a subset
of every kTDS of $G$.
\end{prop}

\begin{prop}
\label{g>=kn/Delta} \emph{(\textbf{Henning, Kazemi \cite{HK11}
2011)}} For any graph $G$ of order $n$ and $\delta (G) \geq k$, $\gamma _{\times k,t}(G)\geq \lceil \frac{kn}{\Delta
(G)}\rceil $.
\end{prop}


\section{Main Results}
\vskip 0.4 true cm

The next proposition is obtained by Proposition \ref{g>=kn/Delta}.

\begin{prop}
\label{L.B.Harary} \emph{i.} $\gamma _{\times k,t}(H_{2m,n}) \geq \lceil
\frac{kn}{2m}\rceil$,

\emph{ii.} $\gamma _{\times k,t}(H_{2m+1,2n} \geq \lceil
\frac{2kn}{2m+1}\rceil$,

\emph{iii.} $\gamma _{\times k,t}(H_{2m+1,2n+1}) \geq \lceil
\frac{k(2n+1)}{2m+2}\rceil $.
\end{prop}

We now calculate the double total domination number of Harary graphs.

\begin{thm}
\label{ H_2m,2n} Let $H_{2m,2n}$ be a Harary graph. Then $\gamma _{\times 2,t}(H_{2m,2n}) =\lceil \frac{n}{m}\rceil.$
\end{thm}

\begin{proof}
Since $S=\{im+1\mid 0\leq i\leq
\lceil \frac{n}{m}\rceil -1\}$ is a 2TDS of $H_{2m,2n}$, Proposition \ref{L.B.Harary} (i) implies that
 $\gamma _{\times 2,t}(H_{2m,2n}) =\lceil \frac{n}{m}\rceil$.
\end{proof}


\begin{thm}
\label{ H_2m+1,2n} Let $2n=(2m+1)\ell+r$, where $0\leq r\leq 2m$, $\ell\geq 1$, and let $\ell+r=2m\ell'+r'$, where $0\leq r'<2m$ and $\ell' \geq 0$. Then
\begin{equation*}
\lceil \frac{4n}{2m+1}\rceil \leq \gamma _{\times 2,t}(H_{2m+1,2n}) \leq \lceil \frac{4n}{2m+1}\rceil +1,
\end{equation*}
if $1\leq r\leq m$ and $(r',\ell')\not\in
\{1,2,...,m\}\times \{0\}$, and $\gamma _{\times 2,t}(H_{2m+1,2n})=\lceil
\frac{4n}{2m+1}\rceil$ otherwise.
\end{thm}

\begin{proof}
Proposition \ref{L.B.Harary} (ii) implies that
\begin{equation*}
\gamma _{\times 2,t}(H_{2m+1,2n}) \geq \lceil
\frac{4n}{2m+1}\rceil =\left\{
\begin{array}{ll}
2\ell & \mbox{if }r=0, \\
2\ell+1 & \mbox{if }1\leq r\leq m, \\
2\ell+2 & \mbox{if }m+1\leq r\leq 2m,
\end{array}%
\right.
\end{equation*}
where $r$ and $\ell$ are both even or are both odd. If $(\ell,r,m)=(1,1,1)$,
then $H_{2m+1,2n}=K_{4}$ and $\gamma _{\times 2,t}(H_{2m+1,2n})=3=\lceil \frac{8}{3}\rceil$.
 Now let $(\ell,r,m) \neq (1,1,1)$. If $r=0$, then the set
\begin{equation*}
S=\{(2m+1)i+1,( 2m+1)i+(m+1) \mid 0\leq i\leq
\ell-1\}
\end{equation*}
is a 2TDS of $H_{2m+1,2n}$ with cardinality $\lceil \frac{4n}{2m+1}\rceil$. Now let $r\neq 0$.
If $\ell'=0$ and $1\leq r'<m$, then the set
\begin{equation*}
S=\{2mi+1,2mi+(m+1) \mid 0\leq i\leq \ell-1\}\cup \{2n-m+1\}
\end{equation*}
is a 2TDS of $G$ of cardinality $2\ell+1=\lceil \frac{4n}{2m+1}\rceil$. Otherwise, for odd $r$, let
\begin{equation*}
\begin{array}{ll}
S= & \{(2m+1)i+(n+m+1),(2m+1)i+(n+2m+2) \mid 0\leq i\leq \lfloor \frac{n}{2m+1}\rfloor -1\}\cup \\
& \{(2m+1)i+1,(2m+1)i+(m+1) \mid 0\leq i\leq \lceil \frac{n}{2m+1}\rceil -1\}\cup \\
& \{n+1,(2m+1)(\lceil \frac{n}{m+1}\rceil -1)+(n+m+1)\},
\end{array}
\end{equation*}
and for even $r$, let
\begin{equation*}
\begin{array}{ll}
S= & \{(2m+1)i+(n+1) ,(2m+1) i+(n+m+1)
\mid 0\leq i\leq \frac{\ell}{2}-1 \}\cup \\
& \{(2m+1)i+1,(2m+1)i+(m+1) \mid 0\leq i\leq \frac{\ell}{2}-1\}\cup
\\
& \{n+1-\frac{r}{2},2n+1-\frac{r}{2}\}.
\end{array}
\end{equation*}
Since in each case the set $S$ is a 2TDS of $H_{2m+1,2n}$ with cardinality $2\ell+2$, our proof is completed.
\end{proof}

\begin{thm}
\label{ H_2m+1,2n+1} Let $2n+1=(2m+1)\ell+r$, where $0\leq r\leq 2m$ and $\ell\geq 1$, and let $\ell+r=2m\ell'+r'$, where $0\leq r'<2m$, and $\ell' \geq 0$. Then
\begin{equation*}
\gamma _{\times 2,t}(H_{2m+1,2n+1})=\lceil \frac{4n+1}{2m+1}\rceil
\end{equation*}
if either $(r,r',\ell')\in \{i\mid 2\leq i\leq m\}\times
\{j\mid 1\leq j\leq m\}\times \{0\}$ or $r\in \{1\}\cup \{i\mid m+2\leq
i\leq 2m\}$, and
\begin{equation*}
\lceil \frac{4n+1}{2m+1}\rceil \leq \gamma _{\times 2,t}(H_{2m+1,2n+1}) \leq \lceil \frac{4n+1}{2m+1}\rceil +1,
\end{equation*}
otherwise.
\end{thm}

\begin{proof}
We first show that $\gamma
_{\times 2,t}(G) \geq \lceil \frac{4n+1}{2m+1}\rceil $. Let $S$ be a 2TDS of $H_{2m+1,2n+1}$ such that $n+1\not\in S$.
Let
\[
t:=\min\{n+1-j\mid j\in S \mbox {, and } n+1-j \mbox { is positive}\}.
\]
Since the set $S'=\{j+t\mid j\in S\}$ is
a 2TDS of $H_{2m+1,2n+1}$ such that $n+1\in S'$ and $\mid S' \mid =\mid
S\mid $, we obtain
\begin{equation*}
\gamma _{\times 2,t}(H_{2m+1,2n+1})=\min \{|S|~|~ S \mbox { is a 2TDS of } H_{2m+1,2n+1} \mbox{ which contains }n+1\}.
\end{equation*}

Now let $S$ be an arbitrary 2TDS of $H_{2m+1,2n+1}$ such that $n+1\in S$. Since
every vertex of $V(G)$ is counted at least two times in the union of the neighborhoods of the vertices of $S$, we have 
\[
\underset{j\in S}{\sum }\deg
(j)\geq 2(2n+1).
\]
Hence $|S| (2m+1)+1\geq 2(2n+1)$, and so $|S| \geq \lceil \frac{4n+1}{2m+1}\rceil $, which implies $\gamma
_{\times 2,t}(G) \geq \lceil \frac{4n+1}{2m+1}\rceil $.

\vspace{0.25 cm}

Now let $2n+1=2m\ell+(\ell+r)$, and let $\ell+r=2m\ell'+r'$, where $0\leq r'<2m$ and $\ell' \geq 0$. We note that $\ell$ is odd if and only if $r$ is even, and
\begin{equation*}
\lceil \frac{4n+1}{2m+1}\rceil =\left \{
\begin{array}{ll}
2\ell & \mbox{if }r=0, \\
2\ell+1 & \mbox{if }1\leq r\leq m, \\
2\ell+2 & \mbox{ if }m+1\leq r\leq 2m.
\end{array}
\right.
\end{equation*}

Let $\alpha _{i}=(2m+1)i$ be an arbitrary vertex of $H_{2m+1,2n+1}$. We continue our proof in the following two cases.

\vspace{0.25 cm}

\textbf{Case 1}. $0\leq r\leq 1$.

For $r=0$, let
\begin{equation*}
\begin{array}{lll}
S_{0} & = & \{\alpha _{i}+1,\alpha _{i}+(m+1) ,\alpha _{i}+(n+1),\alpha _{i}+(n+m+1) \mid 0\leq i\leq \lfloor \frac{\ell}{2}\rfloor -1\}\\
& \cup & \{n-m+1,2n-m+1,2n+1\},
\end{array}
\end{equation*}
and for $r=1$, let
\begin{equation*}
S_{1}=(S_0-\{n-m+1,2n-m+1,2n+1\})\cup \{2n+1\}.
\end{equation*}
Then $S_0$ and $S_1$ are two double total dominating sets of $H_{2m+1,2n+1}$ with cardinality $2\ell+1$.

\vspace{0.25 cm}

\textbf{Case 2. }$2\leq r\leq 2m.$

If $1\leq r'=\ell+r\leq m$, then 
\[
S=\{2mi+1,2mi+(m+1) \mid 0\leq i\leq \ell-1\}\cup\{2n+2-m\}
\]
is a 2TDS of $H_{2m+1,2n+1}$ with cardinality $2\ell+1$. Otherwise, let
\begin{equation*}
\begin{array}{lll}
S_o & = & \{\alpha _{i}+1,\alpha _{i}+(m+1) ,\alpha _{i}+(n+1),\alpha _{i}+(n+m+1) \mid 0\leq i\leq \lfloor \frac{\ell}{2}\rfloor -1\} \\ 
& \cup & \{n+1-(\frac{r-1}{2}),2n+1-(\frac{r-1}{2})\},
\end{array}
\end{equation*}
where $r$ is odd, and let
\begin{equation*}
\begin{array}{lll}
S_e & =& \{\alpha _{i}+1,\alpha _{i}+(m+1),\alpha _{i}+(n+1),\alpha _{i}+(n+m+1) \mid 0\leq i\leq \lfloor \frac{\ell}{2} \rfloor -1\}\\
& \cup & \{n-(\frac{r-2}{2}), n-m-(\frac{r-2}{2}), 2n+1-(\frac{r-2}{2}), 2n+1-m-(\frac{r-2}{2})\}.
\end{array}
\end{equation*}
where $r$ is even. Since the given sets are double total dominating sets of $H_{2m+1,2n+1}$ with cardinality
$\lceil \frac{4n+1}{2m+1}\rceil$ or $\lceil \frac{4n+1}{2m+1}\rceil+1$, our proof is completed.
\end{proof}


\bigskip
\bigskip


{\footnotesize \pn{\bf Adel P. Kazemi}\; \\
{Department of Mathematics}, {University
of Mohaghegh Ardabili, P.O.Box 5619911367,} {Ardabil, Iran}\\
{\tt Email: adelpkazemi@yahoo.com}\\

{\footnotesize \pn{\bf Behnaz Pahlavsay}\; \\ {Department of
Mathematics}, {University of Mohaghegh Ardabili, P.O.Box 5619911367,} {Ardabil, Iran}\\
{\tt Email: pahlavsayb@yahoo.com}

\end{document}